\documentclass{publicadapt}

\usepackage{amsmath}
\usepackage{amssymb}
\usepackage{amsthm}%
\usepackage[OT1]{fontenc}%

\allowdisplaybreaks%

{\theoremstyle{plain}%
 \newtheorem{theorem}{Theorem}
 
 \newtheorem{lemma}{Lemma}%
}

\newcommand{\pfork}{\mathord{\pitchfork}}

{\theoremstyle{remark}
\newtheorem{remark}{Remark}
}

{\theoremstyle{definition}
\newtheorem{definition}{Definition}

}

 \title{On a conjecture on $k$-antichains in the unit $n$-cube}

 \author{John M.\ Campbell}
 \address{Department of Mathematics and Statistics, Dalhousie University, Halifax, Nova Scotia, Canada}
 \email{jh241966@dal.ca}

 \date{June 26, 2026}

 \keywords{$k$-antichain, Hausdorff measure, Stieltjes measure, Lebesgue measure}
 \subjclass{28A78, 05D05}

\begin{document}

 \begin{abstract}
 Let $[0, 1]^{n} \subseteq \mathbb{R}^{n}$ be endowed with its pointwise order, and let $k$ be a positive integer. A subset $A$ of $[0, 1]^{n}$ is said to be a \emph{$k$-antichain} if 
 $\operatorname{card}(A \cap C) \leq k$ for each chain $C \subseteq [0, 1]^{n}$. Letting $\mathcal{H}^{m}$ denote the $m$-dimensional Hausdorff outer measure, Pelekis and Vlas\'{a}k 
 [Publ.\ Math.\ Debrecen, 2020] conjectured that there exists a $k$-antichain $A \subseteq [0, 1]^{n}$ satisfying $\mathcal{H}^{n-1}(A) = k n$, and proved the special case of this 
 conjecture for $n = 2$, whereas Janzer [Mathematika, 2020] proved the $k = 1$ case of Pelekis and Vlas\'{a}k's conjecture. This conjecture is motivated by a result due to Erd\H{o}s on 
 $k$-antichains in $\{ 0, 1 \}^{n}$. We prove Pelekis and Vlas\'{a}k's conjecture in full generality, thus establishing that their upper bound $\mathcal{H}^{n-1}(A) \leq k n$ is sharp for 
 $k$-antichains $A$ in $[0, 1]^{n}$. 
\end{abstract}

\maketitle

\section{Introduction}
 The purpose of this paper is to prove a conjecture that was given by Pelekis and Vlas\'{a}k in 2020 \cite{PelekisVlasak2020}. This conjecture is given after the preliminaries covered below. 

 Let $\text{{\bf x}} = (\text{{\bf x}}_1, \text{{\bf x}}_2, \ldots, \text{{\bf x}}_n)$ and $\text{{\bf y}} = (\text{{\bf y}}_1, \text{{\bf y}}_2, \ldots, \text{{\bf y}}_n)$ be tuples 
 in $\mathbb{R}^{n}$. Define the relation $\leq$ so that $\text{{\bf x}} \leq \text{{\bf y}}$ if $\text{{\bf x}}_i \leq \text{{\bf y}}_i$ holds for all indices $i \in \{ 1, 2, \ldots, n \}$. For 
 $S \subseteq \mathbb{R}^{n}$, a subset $C \subseteq S$ is said to be a \emph{chain} in $S$ if: For all $\text{{\bf x}}, \text{{\bf y}} \in C$, it is the 
 case that $\text{{\bf x}} \leq \text{{\bf y}}$ or $\text{{\bf y}} \leq \text{{\bf x}}$. 

 For a subset $S$ of $\mathbb{R}^{n}$, we adopt the usual notational convention whereby $\operatorname{diam} \, S = \sup\{ \| \text{{\bf x}} - \text{{\bf y}} \| : \text{{\bf x}}, 
 \text{{\bf y}} \in S \}$, for the usual Euclidean norm $\| \cdot \|$ on $\mathbb{R}^{n}$. 
 We then set 
\begin{equation*}
 \mathcal{H}_{\delta}^{s}(A) :=
 \inf\left\{ \sum_{j=1}^{\infty} \frac{\pi^{\frac{s}{2}}}{\Gamma\left( \frac{s}{2} + 1 \right)} 
 \left( \frac{ \operatorname{diam} \, C_{j} }{2} \right)^{s}
 : A \subseteq \bigcup_{j=1}^{\infty} C_{j}, \, \operatorname{diam} \, C_{j} \leq \delta \right\} 
\end{equation*}
 for $A \subseteq \mathbb{R}^{n}$ and $s \in [0, \infty)$ and $\delta \in (0, \infty]$ \cite[p.\ 81]{EvansGariepy2015}. By then setting
\begin{equation}\label{defineHausdorff}
 \mathcal{H}^{s}(A) := \lim_{\delta \to 0} \mathcal{H}_{\delta}^{s}(A),
\end{equation}
 the function defined in \eqref{defineHausdorff} is referred to as the \emph{$s$-dimensional Hausdorff measure on $\mathbb{R}^{n}$}. 
 The definition in \eqref{defineHausdorff} 
 is such that $\mathcal{H}^{0}(\cdot)$ is the counting measure. 

\begin{remark}
 Throughout this paper, \(\mathcal H^s\) denotes the \emph{normalized} Hausdorff measure, so that \(\mathcal H^m\) agrees with the $m$-dimensional Lebesgue measure
 $ \mathcal{L}^m$ on Borel subsets of \(\mathbb R^m\), for every positive integer $m$. 
\end{remark}
 
\begin{definition}\label{defineantichain}
 For a positive integer $k$ and a subset $S$ of $\mathbb{R}^{n}$, a \emph{$k$-antichain} in $S$ is a subset $A \subseteq S$ such that $\mathcal{H}^{0}(A \cap C) \leq k$, for each 
 chain $C \subseteq S$. A 1-antichain may also be referred to as an \emph{antichain}. 
\end{definition}

 As suggested by Pelekis and Vlas\'{a}k \cite{PelekisVlasak2020}, much about the interest in Definition \ref{defineantichain} is given by how this may be seen as giving rise to continuous 
 analogues of problems and results from the field of extremal set theory. In this direction, 
 we begin by observing that statements on families of subsets $\mathcal{F} \subseteq 2^{\{ 1, 2, \ldots, n \}}$
 are equivalent to statements on 
 families of subsets $F \subseteq \{ 0, 1 \}^{n}$, according to the usual correspondence via indicator functions. 
 In this direction, a remarkable result from the discipline of extremal set theory
 concerning $k$-antichains in $\{ 0, 1 \}^{n}$ is due to Erd\H{o}s \cite{Erdos1945} and gives us (in an equivalent way) that:
 If $k \in \{ 1, 2, \ldots, n \}$, and if $A$ is a $k$-antichain in $\{ 0, 1 \}^{n}$, then
\begin{equation}\label{Erdosbound} 
 \mathcal{H}^{0}(A) \leq \sum_{i=1}^{k} \binom{n}{ \left\lfloor \frac{n-k}{2} \right\rfloor + i }. 
\end{equation}
 As noted by Pelekis and Vlas\'{a}k, the bound in \eqref{Erdosbound} is sharp. 

 It was proved by Engel et al.\ \cite{EngelMitsisPelekisReiher2020} that: If $A$ is an antichain in $[0, 1]^{n}$, then
\begin{equation}\label{upperboundn}
 \mathcal{H}^{n-1}(A) \leq n. 
\end{equation}
 Moreover, it was conjectured by Engel et al.\ that there exists an antichain in $[0, 1]^{n}$ satisfying $\mathcal{H}^{n-1}(A) = n$. This conjecture was subsequently proved by Janzer 
 \cite{Janzer2020}. By analogy with \eqref{upperboundn}, it was proved by Pelekis and Vlas\'{a}k \cite{PelekisVlasak2020} that: For each positive integer $k$, if $A$ is a $k$-antichain 
 in $[0, 1]^{n}$, then
\begin{equation}\label{upperkn} 
 \mathcal{H}^{n-1}(A) \leq k n. 
\end{equation}
 Pelekis and Vlas\'{a}k also conjectured that there exists a $k$-antichain in $[0, 1]^{n}$ with a $\mathcal{H}^{n-1}$-measure equal to $kn$, i.e., so that the upper bound in 	 
 \eqref{upperkn} is always sharp. We succeed in proving this conjecture in full generality, whereas Janzer solved the $k = 1$ case and Pelekis and Vlas\'{a}k solved the $n = 2$ case. To the 
 best of our knowledge, this conjecture has remained open, prior to our solution, with regard to extant literature related to 
 the above referenced works by Pelekis and Vlas\'{a}k, by Engel et al., and by Janzer. 
 Our construction is largely based on our extensive interactions with GPT-5.5 Pro. 

\section{Further preliminaries}
 In addition to the preliminaries above, the following preliminaries are necessary for the purposes of our construction, referring to the Evans--Gariepy text 
 \cite{EvansGariepy2015} and further measure theory texts for details. As below, we let $\mathcal B(X)$ 
 denote the 
 Borel $\sigma$-algebra on $X$, for a given subset $X$ of $\mathbb{R}^{d}$. 
 
\begin{definition}\label{formalsingular}
 Let \(\mu\) and \(\nu\) be Borel measures on \(X\subseteq\mathbb R^d\). The measure
 \(\mu\) is said to be \emph{singular} with respect to \(\nu\) 
 if there exists a Borel set \(E\in\mathcal B(X)\) such that
 both $\nu(E)=0$ and 
 $\mu(X\setminus E)=0$ hold. 
 In this case, we write $
\mu\perp\nu$. 
 \end{definition}

 For a vector-valued Radon measure \(\lambda\), we write \(\lambda\perp\nu\) to mean \(|\lambda|\perp\nu\).

\begin{definition}\label{defineLSm}
Let \(I=[a,b]\subseteq\mathbb R\), and let
\(\sigma:I\to\mathbb R\) be nondecreasing and continuous. Extend
 \(\sigma\) to the mapping \(\widetilde\sigma:\mathbb R\to\mathbb R\) defined by
$$
\widetilde\sigma(t)=
\begin{cases}
\sigma(a), & t<a,\\
\sigma(t), & a\le t\le b,\\
\sigma(b), & t>b.
\end{cases}
$$
The \emph{Lebesgue--Stieltjes measure} associated with \(\sigma\), denoted 
 \(\mu_\sigma\) or \(d\sigma\), is the restriction to \(I\) of the unique
Borel measure \(\mu_{\widetilde\sigma}\) on \(\mathbb R\) satisfying
\begin{equation}\label{forseparaterem}
\mu_{\widetilde\sigma}((s,t])
 =
\widetilde\sigma(t)-\widetilde\sigma(s)
\end{equation}
 for $s < t$. 
\end{definition}

\begin{remark}
 Since \(\sigma\), as in Definition \ref{defineLSm}, is continuous, the condition in \eqref{forseparaterem} 
 is equivalent to 
$
\mu_\sigma([s,t])
=
\sigma(t)-\sigma(s)
$ holding for 
$ a \leq s \leq t \leq b$. 
\end{remark}

 As below, we let $C_c^1(\Omega;\mathbb R^m)$ denote the space of continuously differentiable, compactly supported $m$-vector fields on $\Omega$, and we let 
 $ \operatorname{div}\varphi = \sum_{i=1}^m \frac{\partial \varphi_i}{\partial x_i}$.

\begin{definition}\label{defineBV}
 Let $\Omega\subseteq\mathbb R^m$ be open, and let $u\in L^1(\Omega)$. The \emph{variation} of $u$ in $\Omega$ is defined by
$$ V(u,\Omega) := \sup\left\{
\int_{\Omega} u\,\operatorname{div}\varphi \, d\mathcal L^m : \varphi\in C_c^1(\Omega;\mathbb R^m),\
|\varphi(x)|\leq 1 \text{ for all } x\in\Omega
\right\}. $$
We say that $u$ is a \emph{function of bounded variation} on $\Omega$,
and write $ u\in BV(\Omega)$, if $ V(u,\Omega)<\infty$. 
\end{definition}

 As in the below definition, for a given subset $\Omega$ of $\mathbb{R}^{m}$, 
 the expression $C_{c}^{\infty}(\Omega)$ denotes the 
 space of smooth, compactly supported real-valued functions on $\Omega$. 

\begin{definition}
 Let $ \Omega \subseteq \mathbb R^m$ be open, let \(u\in L^1_{\mathrm{loc}}(\Omega)\), and fix \(i\in\{1, 2, \ldots,m\}\). We say that the distributional partial derivative \(D_i u\) 
 is \emph{represented} by a finite signed Radon measure \(\mu_i\)
on \(\Omega\) if
\[
-\int_\Omega u(x)\frac{\partial \varphi}{\partial x_i}(x)\,d\mathcal L^m(x)
=
\int_\Omega \varphi\,d\mu_i
\]
for every \(\varphi\in C_c^\infty(\Omega)\). 
\end{definition}

\begin{definition} 
 Let $u\in BV(\Omega)$. Let $Du$ denote the distributional derivative of $u$. By the Lebesgue decomposition
 theorem, there is a unique decomposition
\begin{equation}\label{defineDs}
Du=\nabla u\,\mathcal L^m+D^{\operatorname{s}} u,
\end{equation}
where $\nabla u\,\mathcal L^m$ is absolutely continuous with respect to
$\mathcal L^m$ and where 
\begin{equation}\label{DsuperpLm} 
 D^{\operatorname{s}} u\perp\mathcal L^m. 
\end{equation}
 The measure $D^{\operatorname{s}} u$ in \eqref{defineDs} is called
 the \emph{singular part of the distributional derivative of $u$}. 
\end{definition}

\subsection{The subgraph perimeter formula for \(BV\) functions}\label{subsec:subgraph-perimeter-formula}
 We make use of a standard perimeter formula for subgraphs of \(BV\) functions.
 This formula is central to the proof of our main result, and we provide a derivation based 
 on a text by Giaquinta et al.\ \cite[\S4.1.5]{GiaquintaModicaSoucek1998}.
 A closely related subgraph characterization of \(BV\) functions traces back 
 to Miranda's work on generalized Cartesian surfaces and sets of locally finite perimeter on Cartesian products \cite{Miranda1964}, 
 but we use a more modern approach based on the cited Giaquinta--Modica--Sou\v{c}ek text
 \cite[\S4.1.5]{GiaquintaModicaSoucek1998}. 
 The main result cited from Giaquinta--Modica--Sou\v{c}ek is stated for base
 dimension \(m\ge2\), and this is the only range in which we shall use the integral formula in \eqref{subgraph-perimeter-formula} below. 

 Let $\Omega\subseteq \mathbb R^m$ be a bounded open set, and let $u$ be a Lebesgue-measurable representative of a function in $ BV(\Omega)$. Being consistent with the notation of 
 Giaquinta--Modica--Sou\v{c}ek \cite[p.\ 371]{GiaquintaModicaSoucek1998}, we define the subgraph of \(u\) by
$
\mathcal S_u
=
\{(x,y)\in\Omega\times\mathbb R:y<u(x)\}$, 
 and we let 
$
\pi:\Omega\times\mathbb R\to\Omega $ 
denote the orthogonal projection onto the first factor.
 If \(\eta\) is a finite signed Radon measure on
\(\Omega\times\mathbb R\), its \emph{pushforward} under \(\pi\), denoted
\(\pi_{\#}\eta\), is the finite signed Radon measure on \(\Omega\) defined by
$
(\pi_{\#}\eta)(B)
=
\eta(\pi^{-1}(B))
=
\eta(B\times\mathbb R)
$
 for every Borel set \(B\subseteq\Omega\). For vector-valued measures, the
 pushforward is understood componentwise.

 Again following Giaquinta--Modica--Sou\v{c}ek \cite[p.\ 371]{GiaquintaModicaSoucek1998}, 
 we define the \(\mathbb R^{m+1}\)-valued measure
\begin{equation}\label{muDu}
\mu_u
=
\mu(Du)
:=
(D_1u, D_2u, \ldots,D_m u,-\mathcal L^m).
\end{equation}
 From the same Giaquinta--Modica--Sou\v{c}ek text \cite[\S4.1.5, Theorem~1]{GiaquintaModicaSoucek1998}, we have the biconditional equivalence
 such that \(u\in BV(\Omega)\)
if and only if \(\mathcal S_u\) has finite perimeter in
\(\Omega\times\mathbb R\). In this case, 
\begin{equation}\label{GMS-pushforward}
\pi_{\#}|D\chi_{\mathcal S_u}|
=
|\mu_u|.
\end{equation}
Equivalently, the equality 
\begin{equation}\label{local-subgraph-perimeter-measure}
|D\chi_{\mathcal S_u}|(B\times\mathbb R)
=
|\mu_u|(B) 
\end{equation}
 holds for every Borel set \(B\subseteq\Omega\). 
 In particular, taking \(B=\Omega\), we obtain
\begin{equation}\label{subgraph-perimeter-vector-measure}
P(\mathcal S_u;\Omega\times\mathbb R)
=
|\mu_u|(\Omega).
\end{equation}

We now rewrite the right-hand side of
\eqref{subgraph-perimeter-vector-measure} in the usual \(BV\) notation. As above, the
Lebesgue decomposition theorem gives us that \eqref{defineDs} holds, 
 again with \eqref{DsuperpLm} holding. 
 So, from \eqref{muDu}, we obtain that 
\begin{equation}\label{applymuDu}
\mu_u
=
(\nabla u,-1)\,\mathcal L^m
+
(D^{\operatorname{s}}u,0).
\end{equation}
 The two measures on the right-hand side of \eqref{applymuDu} are mutually singular. So, 
their total variations add, and we thus derive 
\begin{equation}\label{firstsqrt}
 |\mu_u|
=
\sqrt{1+|\nabla u|^2}\,\mathcal L^m
+
|D^{\operatorname{s}}u| 
\end{equation}
 from \eqref{applymuDu}. 
Combining \eqref{subgraph-perimeter-vector-measure} and \eqref{firstsqrt} then gives
\begin{equation}\label{subgraph-perimeter-formula}
P(\mathcal S_u;\Omega\times\mathbb R)
=
\int_{\Omega}\sqrt{1+|\nabla u|^2}\,d\mathcal L^m
+
|D^{\operatorname{s}}u|(\Omega).
\end{equation}

 We emphasize that \eqref{subgraph-perimeter-formula} is a perimeter formula for the subgraph, not a direct Hausdorff-measure formula for the graph. However, if $u$ is continuous, then 
 the relative topological boundary of
\(\mathcal S_u\) in \(\Omega\times\mathbb R\) is precisely 
 $
\operatorname{graph}u
=
\{(x,u(x)):x\in\Omega\}$. 
Since the reduced boundary is contained in the topological boundary, the
standard structure theorem for sets of finite perimeter gives
$
P(\mathcal S_u;\Omega\times\mathbb R)
=
\mathcal H^m(\partial^*\mathcal S_u\cap(\Omega\times\mathbb R))
\leq
\mathcal H^m(\operatorname{graph}u)$. 
Thus, in the continuous case, the subgraph perimeter formula gives the lower
bound
\begin{equation}\label{subgraph-lower-bound-graph}
\int_{\Omega}\sqrt{1+|\nabla u|^2}\,d\mathcal L^m
+
|D^{\operatorname{s}}u|(\Omega)
\leq
\mathcal H^m(\operatorname{graph}u).
\end{equation}

\section{A full construction}\label{secfull}

\begin{lemma}\label{lem:finite-partials-imply-BV}
Let \(\Omega\subseteq\mathbb R^m\) be open, and let \(u\in L^1(\Omega)\).
Suppose that, for each \(i\in\{1,2,\ldots,m\}\), the distributional partial
derivative \(D_i u\) is represented by a finite signed Radon measure on
\(\Omega\). Then \(u\in BV(\Omega)\) and 
$
Du=(D_1u,D_2u,\ldots,D_m u)
$ 
is a finite \(\mathbb R^m\)-valued Borel measure on \(\Omega\), and
\[
V(u,\Omega)\leq \sum_{i=1}^m |D_i u|(\Omega)<\infty.
\]
\end{lemma}

\begin{proof}
Let $\varphi = (\varphi_1,\varphi_2, \ldots, \varphi_m)\in C_c^1(\Omega;\mathbb R^m)$
satisfy \(|\varphi(x)|\leq 1\) for all \(x\in\Omega\). 
 The distributional identity, initially stated for
 functions in \(C_c^\infty(\Omega)\), extends to \(C_c^1(\Omega)\), by approximating 
 each $\varphi_i$ in $C^1$ by a smooth function with support contained in a
 fixed compact subset of \(\Omega\), and by then using
 the finiteness of $D_i u$ together with $u$ being in $ L^1(\Omega)$. 
 The definition of a distributional partial derivative gives us the equality 
\begin{equation}\label{finitesumpartial}
\int_\Omega u\,\operatorname{div}\varphi\,d\mathcal L^m
=
-\sum_{i=1}^m \int_\Omega \varphi_i\,dD_i u. 
\end{equation} 
 Since \(|\varphi_i|\leq 1\) for $i \in \{ 1, 2, \ldots, m \}$, it follows from \eqref{finitesumpartial} that
\begin{equation}\label{applysup}
\left|
\int_\Omega u\,\operatorname{div}\varphi\,d\mathcal L^m
\right|
\leq
\sum_{i=1}^m |D_i u|(\Omega).
\end{equation}
 Taking the supremum on the left-hand side of \eqref{applysup}
 for $\varphi \in C_c^1(\Omega;\mathbb R^m)$, 
 we obtain that 
$$
V(u,\Omega)\leq \sum_{i=1}^m |D_i u|(\Omega)<\infty, 
$$
 i.e., so that $u \in BV(\Omega)$. Since each component \(D_i u\) is a finite signed
 Radon measure, the tuple $Du = (D_1u, D_2u, \ldots,D_mu)$ is a finite
\(\mathbb R^m\)-valued Borel measure.
\end{proof}

 A full definition of the \emph{Cantor--Lebesgue function} \(\gamma:[0,1]\to[0,1]\)
 involved in the below construction may be found in a number of texts in measure theory and real analysis
 (as in texts by Yeh \cite[pp.\ 76--77]{Yeh2000}
 \cite[pp. 87--88]{Yeh2006}). 
 This provides a standard example of a nondecreasing, continuous, singular function on $[0, 1]$.

\begin{lemma}\label{lemmasigma}
There exists a strictly increasing homeomorphism $ \sigma:[0,1]\to[0,1]$
 such that 
$\sigma(0)=0$ and $\sigma(1)=1$ and 
 whose associated Lebesgue--Stieltjes measure is singular with respect to
one-dimensional Lebesgue measure.
\end{lemma}

\begin{proof}
 As above, we let \(\gamma:[0,1]\to[0,1]\) be the Cantor--Lebesgue function. Now, let
 $K \subseteq [0,1]$ be the middle-thirds Cantor set. Enumerate all nondegenerate
closed intervals with rational endpoints contained in \((0,1)\) as
\begin{equation}\label{countclosed}
 \{[a_r,b_r]\}_{r=1}^{\infty}.
\end{equation}
For each $r$, define 
\begin{equation}\label{gammarx}
\gamma_r(x) =
\begin{cases}
0, & x\le a_r,\\[4pt]
\gamma\!\left(\dfrac{x-a_r}{b_r-a_r}\right),
& a_r\le x\le b_r,\\[10pt]
1, & x\ge b_r.
\end{cases}
\end{equation}
Set
\begin{equation}\label{sigmaseries}
\sigma(x)=\sum_{r=1}^{\infty}2^{-r}\gamma_r(x) 
\end{equation}
 for all $x \in [0, 1]$. 
 The series converges uniformly, giving us that \(\sigma\) is continuous. Since each
\(\gamma_r\) is nondecreasing, it follows that \(\sigma\) is nondecreasing. Moreover, 
 the definition in \eqref{sigmaseries}
 gives that 
 $ \sigma(0)=0$ and that $\sigma(1)=1$ in a direct way. 
 
 Letting $0 \le x < y \le1$, we proceed to fix an index 
$r$ satisfying $
x<a_r<b_r<y$, so that \eqref{gammarx} then gives us that 
 $\gamma_r(x) = 0$ and $\gamma_r(y)=1$. From \eqref{sigmaseries}, 
 we then have that 
$
\sigma(y)-\sigma(x)
\ge
2^{-r}\bigl(\gamma_r(y)-\gamma_r(x)\bigr)
=
2^{-r}>0$, i.e., so that $\sigma$ is strictly increasing. Since it is continuous and maps
\([0,1]\) onto \([0,1]\), it is a homeomorphism. 

It remains to prove that the associated Lebesgue--Stieltjes measure is singular. Let
\begin{equation}\label{defineKr} 
K_r=a_r+(b_r-a_r)K.
\end{equation}
The Lebesgue--Stieltjes measure associated with $\gamma_r$ is supported on
 $K_r$. Moreover, 
 the Lebesgue--Stieltjes measure $\mu_\sigma$ associated with
 $\sigma$ is such that: If $0 \leq s < t \leq 1$, then 
\begin{equation}\label{musigmadiff}
 \mu_\sigma((s,t])
=
\sigma(t)-\sigma(s)
=
\sum_{r=1}^{\infty}2^{-r}
\bigl(\gamma_r(t)-\gamma_r(s)\bigr)
=
\sum_{r=1}^{\infty}2^{-r}\mu_{\gamma_r}((s,t]). 
\end{equation}
 From \eqref{musigmadiff}, we find that 
\begin{equation}\label{musigmagamma} 
\mu_\sigma
=
\sum_{r=1}^{\infty}2^{-r}\mu_{\gamma_r}.
\end{equation}
 Now, define 
\begin{equation}\label{Ebigcup}
 E = \bigcup_{r=1}^{\infty}K_r.
\end{equation}
 From \eqref{defineKr}, we find that 
 $K_r$ has Lebesgue measure zero for each positive integer $r$. 
 Consequently, from \eqref{Ebigcup} 
 together with the countable subadditivity of measures, the vanishing 
$ \mathcal L^1(E)=0 
$ holds. On the other hand, each measure \(\mu_{\gamma_r}\) is supported on \(K_r\), and (using the expansion in \eqref{musigmagamma})
 we can then deduce that $\mu_\sigma$ is supported on $E$. This gives us that 
$\mu_\sigma([0,1]\setminus E)=0$. 
 This together with the vanishing of $\mathcal{L}^{1}(E)$ together give us  that $ \mu_\sigma\perp \mathcal L^1$.
\end{proof} 

 We henceforth assume that $n \geq 2$ unless otherwise specified, writing $m = n - 1$. 
 Our strategy is to construct an antichain in $[0, 1]^{n}$ of $\mathcal{H}^{n-1}$-measure $n$, 
 and then form $k$ antichains based on our original antichain. 

 Now, let $\sigma$ be a function satisfying the conditions in Lemma \ref{lemmasigma}. For a parameter $a > 0$, and for a tuple 
 $ \text{{\bf t}} = (\text{{\bf t}}_1, \text{{\bf t}}_2, \ldots, \text{{\bf t}}_m) \in (0, 1)^{m}$, let 
\begin{equation}\label{Phiat} 
 \Phi_{a}( \text{{\bf t}} ) = 
 \frac{a \prod_{j=1}^{m} (1 - \text{{\bf t}}_j)}{ a \prod_{j=1}^{m} (1 - \text{{\bf t}}_j) + \prod_{j=1}^{m} \text{{\bf t}}_j}.
\end{equation}
 By rewriting the right-hand side of \eqref{Phiat} so that 
 $$ \Phi_{a}(\text{{\bf t}}) = \left( 1 + \frac{1}{a} \prod_{j=1}^{m} \frac{\text{{\bf t}}_j}{1 - \text{{\bf t}}_{j}} \right)^{-1}, $$
 we then find that $\Phi_{a}(\text{{\bf t}})$ is strictly decreasing in each coordinate. 
 We proceed to set 
\begin{equation}\label{defineua}
 u_{a}(\text{{\bf x}}_1, \text{{\bf x}}_2, \ldots, \text{{\bf x}}_{m}) = 
 \Phi_{a}\big( \sigma(\text{{\bf x}}_{1}), \sigma(\text{{\bf x}}_{2}), \ldots, \sigma(\text{{\bf x}}_{m}) \big) 
\end{equation}
 for a given element $\text{{\bf x}}$ in $(0, 1)^{m}$.
 We then define
\begin{equation}\label{Gammaa}
 \Gamma_{a} = \left\{ \big( \text{{\bf x}}_1, \text{{\bf x}}_2, \ldots, \text{{\bf x}}_{m}, 
 u_{a}(\text{{\bf x}}_1, \text{{\bf x}}_2, \ldots, \text{{\bf x}}_{m}) \big) : \text{{\bf x}} \in (0, 1)^{m} \right\} 
 \subseteq (0, 1)^{n}. 
\end{equation}
 For convenience, we may rewrite the tuple $ \big( \text{{\bf x}}_1, \text{{\bf x}}_2, \ldots, \text{{\bf x}}_{m}, 
 u_{a}(\text{{\bf x}}_1, \text{{\bf x}}_2, \ldots, \text{{\bf x}}_{m}) \big) $ shown in 
 \eqref{Gammaa} as $(\text{{\bf x}}, u_{a}(\text{{\bf x}}))$. 

\begin{lemma}\label{oneantichain}
 For each $a > 0$, we have that $\Gamma_a$ is an antichain. 
\end{lemma}

\begin{proof}
 Let $\text{{\bf p}} = (\text{{\bf x}}, u_{a}(\text{{\bf x}}))$ and $\text{{\bf q}} = (\text{{\bf y}}, u_a(\text{{\bf y}}))$ be two unequal points in 
 $\Gamma_a$. Suppose that $\text{{\bf p}} \leq \text{{\bf q}}$
 in the pointwise order. Consequently, the relation $\text{{\bf x}}_i \leq \text{{\bf y}}_i$ holds
 for $i \in \{ 1, 2, \ldots m \}$, with at least one strict inequality unless $\text{{\bf x}} = \text{{\bf y}}$. 

 Since $\sigma$ (as in Lemma \ref{lemmasigma}) is strictly increasing, and since $\Phi_a$
 is strictly decreasing in each coordinate, we see that: If $\text{{\bf x}} < \text{{\bf y}}$, then 
\begin{equation}\label{strictcontradict}
 u_{a}(\text{{\bf x}}) > u_a(\text{{\bf y}}). 
\end{equation}
 However, since $\text{{\bf p}} \leq \text{{\bf q}}$,
 a comparison of the final coordinates of $\text{{\bf p}}$ and $\text{{\bf q}}$ then gives us that 
 $u_{a}(\text{{\bf x}}) \leq u_a(\text{{\bf y}})$, contradicting \eqref{strictcontradict}. 
 Since no two distinct points in $\Gamma_a$ are comparable, we obtain the desired result. 
\end{proof}

 We next prove that $u_a$ is in $BV\big( (0,1)^m \big)$, and we compute 
 the singular part of its distributional derivative. After this has been established, we apply the subgraph perimeter 
 formula from Subsection~\ref{subsec:subgraph-perimeter-formula}, together with the antichain upper bound, to compute \(\mathcal H^m(\Gamma_a)\).
 As below, we let the operation given by the concatenation of tuples be denoted as $\cdot$. 

\begin{lemma}\label{lem:slicing-partial-variation}
 Let $\Omega=(0,1)^m$, and fix \(i\in\{1,2,\ldots,m\}\). Define $ Z=(0,1)^{m-1}$, with the convention  whereby  $Z=(0,1)^0$ is interpreted as a
 one-point measure space of mass $1$. For $  \text{{\bf z}} = (\mathbf z_1,\mathbf z_2,\ldots,\mathbf z_{m-1}) \in Z $
and \(s\in(0,1)\), define
\begin{equation}\label{defineiota}
\iota_i(s,\mathbf z)
= (\mathbf z_1,\mathbf z_2,\ldots,\mathbf z_{i-1}) \cdot 
(s) \cdot ( \text{{\bf z}}_i,\mathbf z_{i+1},\ldots,\mathbf z_{m-1}),
\end{equation}
  omitting empty tuples              for  $i = 1 \wedge i=m$ or
\(m=1\). For a Borel set \(B\subseteq\Omega\), define
$
B_{\mathbf z}^{i}
=
\{s\in(0,1):\iota_i(s,\mathbf z)\in B\}$.
 Let \(u\in L^1(\Omega)\). Define 
\begin{equation}\label{wbfs}
 w_{\mathbf z}(s)=u(\iota_i(s,\mathbf z)). 
\end{equation}
 Suppose that there exists a Borel set
\(Z_0\subseteq Z\) such that $ \mathcal L^{m-1}(Z\setminus Z_0)=0 $ and such that:  For each \(\mathbf z\in Z_0\), we have that 
 $w_{\mathbf z}(s) $ 
belongs to \(BV((0,1))\), and its distributional derivative \(Dw_{\mathbf z}\)
is a nonpositive finite signed Radon measure on \((0,1)\).
For \(\mathbf z\in Z_0\), set
$
\lambda_{\mathbf z}=Dw_{\mathbf z}$, 
and, for $ \text{{\bf z}} \in Z\setminus Z_0$, set
$
\lambda_{\mathbf z}=0$.
Assume that the family \(\{\lambda_{\mathbf z}\}_{\mathbf z\in Z}\) is
measurable in the sense that, for every Borel set \(A\subseteq(0,1)\), the map
$
\mathbf z\mapsto \lambda_{\mathbf z}(A) $
 is $\mathcal{L}^{m-1}$-measurable. Assume, moreover, that
\begin{equation}\label{slicing-integrability-condition}
\int_Z
|\lambda_{\boldsymbol\eta}|((0,1))
\,d\mathcal L^{m-1}(\boldsymbol\eta)
<\infty.
\end{equation}
 Then 
$$
-\int_\Omega u(x)\frac{\partial\varphi}{\partial x_i}(x)\,
d\mathcal L^m(x)
=
\int_\Omega \varphi\,d\nu_i
$$
for every \(\varphi\in C_c^\infty(\Omega)\), and   
\begin{equation}\label{slicing-signed-formula}
\nu_i(B) =   \int_Z
\lambda_{\boldsymbol\eta}(B_{\boldsymbol\eta}^{i})
\,d\mathcal L^{m-1}(\boldsymbol\eta) 
\end{equation}
 for every Borel set
$B\subseteq\Omega$.  Moreover,
\begin{equation}\label{slicing-variation-formula}
|\nu_i|(B)
=
\int_Z
|\lambda_{\boldsymbol\eta}|(B_{\boldsymbol\eta}^{i})
\,d\mathcal L^{m-1}(\boldsymbol\eta).
\end{equation}
Equivalently, after identifying \(D_i u\) with the measure \(\nu_i\), one may
write \(D_i u(B)\) in place of \(\nu_i(B)\) in
\eqref{slicing-signed-formula}.
\end{lemma}

\begin{proof}
  For $ \text{{\bf z}} \in Z_0$, we have
  that \(\lambda_{\mathbf z}=Dw_{\mathbf z}\), while for
\(\mathbf z\in Z\setminus Z_0\), we have that \(\lambda_{\mathbf z}=0\).
 For a Borel set
$B \subseteq\Omega$, the section
\[
B_{\mathbf z}^{i}
=
\{s\in(0,1):\iota_i(s,\mathbf z)\in B\}
\]
is a Borel subset of \((0,1)\). Moreover, by the measurability assumption on
the family \(\{\lambda_{\mathbf z}\}_{\mathbf z\in Z}\), together with the
usual monotone class argument starting from Borel rectangles in \(\Omega\), the
maps
\[
\mathbf z\mapsto \lambda_{\mathbf z}(B_{\mathbf z}^{i})
\quad\text{and}\quad
\mathbf z\mapsto |\lambda_{\mathbf z}|(B_{\mathbf z}^{i})
\]
are \(\mathcal L^{m-1}\)-measurable.

Define a signed set function \(\nu_i\) on the Borel subsets of \(\Omega\) by
\[
\nu_i(B)
=
\int_Z
\lambda_{\boldsymbol\eta}(B_{\boldsymbol\eta}^{i})
\,d\mathcal L^{m-1}(\boldsymbol\eta).
\]
Since each \(\lambda_{\boldsymbol\eta}\) is nonpositive, \(\nu_i\) is also
nonpositive, and
\[
-\nu_i(B)
=
\int_Z
|\lambda_{\boldsymbol\eta}|(B_{\boldsymbol\eta}^{i})
\,d\mathcal L^{m-1}(\boldsymbol\eta)
\]
for every Borel set \(B\subseteq\Omega\).

We next check countable additivity. Let \(\{B_r\}_{r=1}^{\infty}\) be pairwise
disjoint Borel subsets of \(\Omega\). Then, for each
\(\boldsymbol\eta\in Z\),
\[
\left(\bigcup_{r=1}^{\infty}B_r\right)_{\boldsymbol\eta}^{i}
=
\bigcup_{r=1}^{\infty}(B_r)_{\boldsymbol\eta}^{i},
\]
and the union on the right is pairwise disjoint. Since
\(|\lambda_{\boldsymbol\eta}|\) is a positive measure, monotone convergence
gives
\[
\begin{aligned}
-\nu_i\left(\bigcup_{r=1}^{\infty}B_r\right)
&=
\int_Z
|\lambda_{\boldsymbol\eta}|
\left(
\left(\bigcup_{r=1}^{\infty}B_r\right)_{\boldsymbol\eta}^{i}
\right)
\,d\mathcal L^{m-1}(\boldsymbol\eta) \\[4pt]
&=
\int_Z
\sum_{r=1}^{\infty}
|\lambda_{\boldsymbol\eta}|((B_r)_{\boldsymbol\eta}^{i})
\,d\mathcal L^{m-1}(\boldsymbol\eta) \\[4pt]
&=
\sum_{r=1}^{\infty}
\int_Z
|\lambda_{\boldsymbol\eta}|((B_r)_{\boldsymbol\eta}^{i})
\,d\mathcal L^{m-1}(\boldsymbol\eta) \\[4pt]
&=
\sum_{r=1}^{\infty}(-\nu_i(B_r)).
\end{aligned}
\]
Thus \(-\nu_i\) is a positive Borel measure, and hence \(\nu_i\) is a signed
 Radon measure. Also,
\[
-\nu_i(\Omega)
=
\int_Z
|\lambda_{\boldsymbol\eta}|((0,1))
\,d\mathcal L^{m-1}(\boldsymbol\eta)
<\infty
\]
by the assumed integrability condition. Therefore \(\nu_i\) is a finite signed
 Radon measure. Since \(\Omega\) is an open subset of Euclidean space, the
standard regularity theorem for finite Borel measures gives that \(\nu_i\) is a
finite signed Radon measure.

Because \(\nu_i\) is nonpositive, its total variation is \(|\nu_i|=-\nu_i\).
Consequently,  
 the relation $$
|\nu_i|(B)
=
-\nu_i(B)
=
\int_Z
|\lambda_{\boldsymbol\eta}|(B_{\boldsymbol\eta}^{i})
\,d\mathcal L^{m-1}(\boldsymbol\eta) 
$$
 holds for each Borel set \(B\subseteq\Omega\). 
 This proves the asserted total variation formula for \(\nu_i\).

It remains to identify \(\nu_i\) with the distributional partial derivative
\(D_i u\). Let \(\varphi\in C_c^\infty(\Omega)\). By the definition of
\(\nu_i\), first for simple functions and then by approximation using the
integrability of
\(\boldsymbol\eta\mapsto |\lambda_{\boldsymbol\eta}|((0,1))\), we have
\[
\int_\Omega \varphi\,d\nu_i
=
\int_Z
\int_{(0,1)}
\varphi(\iota_i(s,\boldsymbol\eta))
\,d\lambda_{\boldsymbol\eta}(s)
\,d\mathcal L^{m-1}(\boldsymbol\eta).
\]
Since \(Z\setminus Z_0\) has \(\mathcal L^{m-1}\)-measure zero and
\(\lambda_{\boldsymbol\eta}=0\) on \(Z\setminus Z_0\), this equals
\[
\int_{Z_0}
\int_{(0,1)}
\varphi(\iota_i(s,\boldsymbol\eta))
\,dDw_{\boldsymbol\eta}(s)
\,d\mathcal L^{m-1}(\boldsymbol\eta).
\]
For each \(\boldsymbol\eta\in Z_0\), the function
$
s\mapsto \varphi(\iota_i(s,\boldsymbol\eta))
$ 
belongs to \(C_c^\infty((0,1))\). Hence, by the one-dimensional definition of
the distributional derivative of \(w_{\boldsymbol\eta}\),
\[
\begin{aligned}
\int_\Omega \varphi\,d\nu_i
&=
-\int_{Z_0}
\int_{(0,1)}
w_{\boldsymbol\eta}(s)
\frac{\partial}{\partial s}
\bigl(\varphi(\iota_i(s,\boldsymbol\eta))\bigr)
\,ds\,d\mathcal L^{m-1}(\boldsymbol\eta) \\[4pt]
&=
-\int_{Z_0}
\int_{(0,1)}
u(\iota_i(s,\boldsymbol\eta))
\frac{\partial\varphi}{\partial x_i}
(\iota_i(s,\boldsymbol\eta))
\,ds\,d\mathcal L^{m-1}(\boldsymbol\eta).
\end{aligned}
\]
Since \(Z\setminus Z_0\) has \(\mathcal L^{m-1}\)-measure zero, Fubini's theorem
then gives
\[
\begin{aligned}
\int_\Omega \varphi\,d\nu_i
&=
-\int_Z
\int_{(0,1)}
u(\iota_i(s,\boldsymbol\eta))
\frac{\partial\varphi}{\partial x_i}
(\iota_i(s,\boldsymbol\eta))
\,ds\,d\mathcal L^{m-1}(\boldsymbol\eta) \\[4pt]
&=
-\int_\Omega
u(\mathbf x)\frac{\partial\varphi}{\partial x_i}(\mathbf x)
\,d\mathcal L^m(\mathbf x).
\end{aligned}
\]
The final expression is \(D_i u(\varphi)\). Therefore
\[
D_i u(\varphi)=\int_\Omega \varphi\,d\nu_i
\qquad
\text{for every } \varphi\in C_c^\infty(\Omega).
\]
Thus \(D_i u=\nu_i\) as distributions. Hence \(D_i u\) is represented by the
finite signed Radon measure \(\nu_i\), and the signed and total variation
formulas are precisely
\[
\nu_i(B)
=
\int_Z
\lambda_{\boldsymbol\eta}(B_{\boldsymbol\eta}^{i})
\,d\mathcal L^{m-1}(\boldsymbol\eta)
\]
and
\[
|\nu_i|(B)
=
\int_Z
|\lambda_{\boldsymbol\eta}|(B_{\boldsymbol\eta}^{i})
\,d\mathcal L^{m-1}(\boldsymbol\eta),
\]
for every Borel set \(B\subseteq\Omega\).
\end{proof}

\begin{lemma}\label{afterintegral}
 For each $a > 0$, the function \(u_a\) belongs to \(BV((0,1)^m)\), the measure
\(Du_a\) is singular with respect to \(\mathcal L^m\), and
\[
|D^{\operatorname{s}}u_a|\big((0,1)^m\big)=m.
\]
\end{lemma}

\begin{proof}
 We again let $u = u_a\colon (0, 1)^{m} \to \mathbb{R}$ be defined as in 
 \eqref{defineua}. Now, let 
\begin{equation}\label{displaytupx}
 \text{{\bf x}} = (\text{{\bf x}}_{1}, \text{{\bf x}}_{2}, \ldots, \text{{\bf x}}_{m})
\end{equation}
 be an element in $(0, 1)^{m}$, and let $i \in \{ 1, 2, \ldots, m \}$. 
 We then define
\begin{equation}\label{displaytupz}
 \text{{\bf z}} = \text{{\bf z}}(\text{{\bf x}}, i) 
 = (\text{{\bf x}}_{1}, \text{{\bf x}}_{2}, \ldots, \text{{\bf x}}_{i-1}, \text{{\bf x}}_{i+1}, 
 \text{{\bf x}}_{i+2}, \ldots, \text{{\bf x}}_{m}). 
\end{equation}
 Now, define
\begin{equation}\label{Lambdadomain}
 \Lambda_{i,\mathbf x}\colon (0,1)\to(0,1)
\end{equation}
so that, for \(s\in(0,1)\),
\begin{equation}\label{defineLambdaix}
 \Lambda_{i,\mathbf x}(s) =
 u_a\big(
 \mathbf x_1,\mathbf x_2,\ldots,\mathbf x_{i-1},
 s,
 \mathbf x_{i+1},\mathbf x_{i+2},\ldots,\mathbf x_m
 \big).
\end{equation}
 The definitions in \eqref{defineua} and 
 \eqref{defineLambdaix} together give us that $ \Lambda_{i, \text{{\bf x}}}(s) $ $ = $ $ \Phi_{a}\big(
 \sigma\left(\text{{\bf x}}_{1}\right)$, $ \sigma\left( \text{{\bf x}}_{2} \right)$, $ \ldots$, 
 $ \sigma\left( \text{{\bf x}}_{i-1} \right)$, $ \sigma\left( s \right)$, $ 
 \sigma\left( \text{{\bf x}}_{i+1} \right)$, $ \sigma\left( \text{{\bf x}}_{i+2} \right)$, $ \ldots$, $ \sigma\left( \text{{\bf x}}_{m} \right)
 \big)$, 
 again for $\sigma $ as in Lemma \ref{lemmasigma}. In a similar spirit, we define 
\begin{equation}\label{psidomain}
 \psi_{i, \text{{\bf x}}}\colon (0, 1) \to (0, 1) 
\end{equation}
 so that: For a member $t$ of the domain in \eqref{psidomain}, we have that 
\begin{equation}\label{definepsi}
 \psi_{i, \text{{\bf x}}}(t) = \Phi_{a}\left(
 \sigma\left(\text{{\bf x}}_{1}\right), \sigma\left( \text{{\bf x}}_{2} \right), \ldots, 
 \sigma\left( \text{{\bf x}}_{i-1} \right), t, 
 \sigma\left( \text{{\bf x}}_{i+1} \right), \sigma\left( \text{{\bf x}}_{i+2} \right), \ldots, \sigma\left( \text{{\bf x}}_{m} \right)
 \right). 
\end{equation}
 By then setting 
\begin{equation}\label{capA} 
 A = A_{i, \text{{\bf x}}} = a \prod_{j \neq i} \left( 1 - \sigma\left( \text{{\bf x}}_{j} \right) \right) > 0 
\end{equation}
 and 
\begin{equation}\label{capB} 
 B = B_{i, \text{{\bf x}}} = \prod_{j \neq i} \sigma\left( \text{{\bf x}}_{j} \right) > 0, 
\end{equation}
 a combined application of the definitions among \eqref{Phiat}, \eqref{definepsi}, 
 \eqref{capA}, and \eqref{capB} gives us that 
\begin{equation}\label{psiAB} 
 \psi_{i, \text{{\bf x}}}(t) = \frac{A(1-t)}{A(1-t) + B t}. 
\end{equation}
 From \eqref{psiAB}, we find that 
\begin{equation}\label{negderivative}
 \psi_{i, \text{{\bf x}}}(t)' = - \frac{AB}{\left( A(1-t) + B t \right)^2} < 0. 
\end{equation}
 We then find that $\psi_{i, \text{{\bf x}}}$
 is strictly decreasing on its domain, from \eqref{negderivative}. 
 So, since $\sigma$ is strictly increasing and continuous, 
 we find that 
 $\psi_{i, \text{{\bf x}}}(\sigma(s))$ is continuous and strictly decreasing in $s$, 
 i.e., so that $\Lambda_{i, \text{{\bf x}}}(s)$ is continuous and strictly decreasing in $s$. 
 Also observe that $\lim_{t \to 0^{+}} \psi_{i, \text{{\bf x}}}(t) = 1$
 and that $\lim_{t \to 1^{-}} \psi_{i, \text{{\bf x}}}(t) = 0$, with $\sigma(0) = 0$
 and $\sigma(1) = 1$, so that 
 $\lim_{s \to 0^{+}} \Lambda_{i, \text{{\bf x}}}(s) = 1$
 and $\lim_{s \to 1^{-}} \Lambda_{i, \text{{\bf x}}}(s) = 0$. 

 This gives that the one-dimensional total variation for $\Lambda_{i, \text{{\bf x}}}$ over its domain
 satisfies
\begin{equation}\label{onedimensional}
 \left| D\Lambda_{i, \text{{\bf x}}} \right|\big( (0, 1) \big) = 1. 
\end{equation}
 Now, write \(\mu=d\sigma\). For each
\(\mathbf z=(\mathbf z_1,\ldots,\mathbf z_{m-1})\in(0,1)^{m-1}\), define
$
\pfork_{i,\mathbf z}\colon(0,1)\to(0,1)
$ (cf.\ \eqref{definepsi})
by
$$
\pfork_{i,\mathbf z}(t)
=
\Phi_a\left(
\sigma(\mathbf z_1),\ldots,\sigma(\mathbf z_{i-1}) \right) \cdot \left( 
t\right) \cdot \left( 
\sigma(\mathbf z_i),\ldots,\sigma(\mathbf z_{m-1})
\right),
$$
   omitting empty tuples in the manner suggested previously. 
 We also define 
\begin{equation}\label{defineomega}
\omega_{i,\mathbf z}(s)
=
u_a(\iota_i(s,\mathbf z)) 
\end{equation}
 for $s\in(0,1)$. We may rewrite the composition on the right-hand side of \eqref{defineomega} so that 
\begin{equation}\label{omegaforksigma}
\omega_{i,\mathbf z}
=
\pfork_{i,\mathbf z}\circ\sigma.
\end{equation}
 Moreover, if \(\mathbf z=\mathbf z(\mathbf x,i)\), then
$
\pfork_{i,\mathbf z}=\psi_{i,\mathbf x}$ and $ \omega_{i,\mathbf z}=\Lambda_{i,\mathbf x}$. 
 Now, from \eqref{omegaforksigma}, 
the one-dimensional chain rule for Lebesgue--Stieltjes measures gives
\begin{equation}\label{Domega-chain}
D\omega_{i,\mathbf z}
=
(\pfork_{i,\mathbf z}'\circ\sigma)\,\mu .
\end{equation}
 Indeed, for \(0<s<t<1\),
\[
\int_{(s,t]} \pfork_{i,\mathbf z}'(\sigma(r))\,d\mu(r)
=
\int_{\sigma(s)}^{\sigma(t)}\pfork_{i,\mathbf z}'(q)\,dq
=
\pfork_{i,\mathbf z}(\sigma(t))-\pfork_{i,\mathbf z}(\sigma(s)),
\]
and equality on intervals determines the corresponding Lebesgue--Stieltjes
measure.

Thus, for every Borel set \(A\subseteq(0,1)\),
\[
D\omega_{i,\mathbf z}(A)
=
\int_A \pfork_{i,\mathbf z}'(\sigma(s))\,d\mu(s).
\]
By the displayed formula for \(\psi_{i,\mathbf x}'\), and by the analogous
formula for \(\pfork_{i,\mathbf z}'\), the map
$
(\mathbf z,s)\mapsto \pfork_{i,\mathbf z}'(\sigma(s))
$
is Borel measurable.
 Since \(\pfork_{i,\mathbf z}'<0\), we also have that 
$
|D\omega_{i,\mathbf z}|
=
-(\pfork_{i,\mathbf z}'\circ\sigma)\,\mu$.
 Consequently, the maps
$
\mathbf z\mapsto D\omega_{i,\mathbf z}(A)$ and $
\mathbf z\mapsto |D\omega_{i,\mathbf z}|(A)$
 are \(\mathcal L^{m-1}\)-measurable for every Borel set
\(A\subseteq(0,1)\).

 The preceding discussion shows that: For each $ \text{{\bf z}} \in (0,1)^{m-1}$, we have that 
$
\omega_{i,\mathbf z}(s)=u_a(\iota_i(s,\mathbf z))
$
is continuous and decreasing. Hence \(D\omega_{i,\mathbf z}\) is a nonpositive finite Radon measure on
\((0,1)\).

Given \(\mathbf z\in(0,1)^{m-1}\), choose any
$\mathbf x\in(0,1)^m$ such that
$
\mathbf z=\mathbf z(\mathbf x,i)$. 
Then
$
\omega_{i,\mathbf z}=\Lambda_{i,\mathbf x}$. Now, by \eqref{onedimensional},
\begin{equation}\label{onedimensional-omega}
|D\omega_{i,\mathbf z}|((0,1))=1.
\end{equation}
for every \(\mathbf z\in(0,1)^{m-1}\).
 Moreover,
\[
\int_{(0,1)^{m-1}}
|D\omega_{i,\mathbf z}|((0,1))\,d\mathcal L^{m-1}(\mathbf z)
=
\int_{(0,1)^{m-1}}1\,d\mathcal L^{m-1}(\mathbf z)
=
1,
\]
with the convention that \((0,1)^0\) has mass \(1\) when \(m=1\).

Since $0 < u_a < 1$,   we have that   $  u_a$ is in $ L^1((0,1)^m)$. Consequently, Lemma \ref{lem:slicing-partial-variation} applies to \(u=u_a\), with
$$
Z_0=(0,1)^{m-1},
\qquad
w_{\mathbf z}=\omega_{i,\mathbf z},
\qquad   \lambda_{\mathbf z}=D\omega_{i,\mathbf z}. $$ Consequently, the equality 
\begin{equation}\label{sliceformula-Diua}
|D_i u_a|(B) = \int_{(0,1)^{m-1}}
|D\omega_{i,\mathbf z}|
\bigl(\{s\in(0,1):\iota_i(s,\mathbf z)\in B\}\bigr)
\,d\mathcal L^{m-1}(\mathbf z).
\end{equation}
 holds for every Borel set $B\subseteq(0,1)^m$.  Taking $B = (0,1)^m$ in  \eqref{sliceformula-Diua}, we obtain
\begin{equation}\label{latersummand}
\left|D_i u_a\right|\big((0,1)^m\big) = \int_{(0,1)^{m-1}}
|D\omega_{i,\mathbf z}|((0,1))
\,d\mathcal L^{m-1}(\mathbf z)  = 1. 
\end{equation}
 In particular, for each $i \in \{1, 2, \ldots, m\}$, the distributional partial derivative \(D_i u_a\) is represented by a finite signed Radon measure on $(0,1)^m$. Also, 
 since $0 < u_a < 1$, we find that $u_a \in L^1\big( (0,1)^m \big)$. Lemma \ref{lem:finite-partials-imply-BV} thus implies that
$ u_a$ is in $ BV\big( (0,1)^m \big)$. Consequently, the distributional derivative 
 $$ Du_a=(D_1u_a,D_2u_a,\ldots,D_m u_a) $$
is a finite \(\mathbb R^m\)-valued Borel measure on \((0,1)^m\).

 We write $\mu = d \sigma$ in place of the Lebesgue--Stieltjes measure corresponding to $\sigma$. From Lemma \ref{lemmasigma}, we have that $\mu \perp \mathcal{L}^{1}$. So, there 
 exists a Borel set $E \subseteq (0, 1)$ such that both $\mathcal{L}^{1}(E) = 0$ and 
\begin{equation}\label{muconcentrated} 
 \mu\big( (0, 1) \setminus E \big) = 0
\end{equation}
 hold. 

Define
\begin{equation}\label{threeCartesian}
S_i=(0,1)^{i-1}\times E\times(0,1)^{m-i}.
\end{equation}
 Since \(\mathcal L^1(E)=0\), Tonelli's theorem gives
\begin{equation}\label{LmSi0}
\mathcal L^m(S_i)=0.
\end{equation}
 Using \eqref{sliceformula-Diua} with \(B=(0,1)^m\setminus S_i\), we obtain that 
$$ \begin{aligned}
|D_i u_a|((0,1)^m\setminus S_i)
&=
\int_{(0,1)^{m-1}}
|D\omega_{i,\mathbf z}|
\bigl(\{s\in(0,1):\iota_i(s,\mathbf z)\notin S_i\}\bigr)
\,d\mathcal L^{m-1}(\mathbf z) \\[4pt]
&=
\int_{(0,1)^{m-1}}
|D\omega_{i,\mathbf z}|((0,1)\setminus E)
\,d\mathcal L^{m-1}(\mathbf z)
=
0.
\end{aligned}
$$ So, we find that $|D_i u_a|$ is concentrated on $S_i$. Since \(\mathcal L^m(S_i)=0\), it follows that $ D_i u_a\perp\mathcal L^m$. 

 Now, let $i \neq j$. We also let $$ q = q(i, j) = \begin{cases} 
 j, & \text{if $j < i$,} \\
 j - 1, & \text{if $j > i$.} 
 \end{cases} $$
 This gives us that 
\begin{equation}\label{iotaequiv} 
 \iota_{i}(s, \text{{\bf z}}) \in S_{j} \Longleftrightarrow \text{{\bf z}}_{q} \in E. 
\end{equation} 
 The biconditional equivalence in \eqref{iotaequiv} then gives us that $$ \big\{ s \in (0, 1) : \iota_{i}(s, \text{{\bf z}}) \in S_{j} \big\} = \begin{cases} 
 (0, 1), & \text{if $\text{{\bf z}}_{q} \in E$,} \\
 \varnothing, & \text{if $\text{{\bf z}}_{q} \not\in E$.} 
 \end{cases} $$
 Using \eqref{sliceformula-Diua} with \(B=S_j\), we find that $$\left| D_{i}u_{a} \right|(S_{j}) = \int_{ \{ \text{{\bf z}} \in (0, 1)^{m-1} : \text{{\bf z}}_{q} \in E \} } \left| D \omega_{i, 
 \text{{\bf z}}} \right|\big( (0, 1) \big) d\mathcal{L}^{m-1}(\text{{\bf z}}).$$ 

 Since \eqref{onedimensional-omega} holds for every $\mathbf z\in(0,1)^{m-1}$, this gives us that $|D_{i} u_{a}|(S_{j}) = \mathcal{L}^{m-1}\big( \{ \text{{\bf z}} \in (0, 1)^{m - 
 1} : \text{{\bf z}}_{q} \in E \} \big)$. Since $ \big\{ \text{{\bf z}} \in (0, 1)^{m-1} : \text{{\bf z}}_{q} \in E \big\} = (0, 1)^{q-1} \times E \times (0, 1)^{m-1-q}$, and since an application 
 of Tonelli's theorem gives us that $$\mathcal{L}^{m-1}\big( (0, 1)^{q-1} \times E \times (0, 1)^{m-1-q} \big)
 = \mathcal{L}^{1}(E) = 0,$$ 
 we find that $\left| D_{i}u_a \right|(S_{j}) = 0$, i.e., if $i \neq j$. 
 So, since $\left| D_{j} u_a \right|$ is concentrated on $S_{j}$
 and since $\left| D_{i} u_{a} \right|(S_{j}) = 0$, 
 we find that $\left| D_{i} u_{a} \right| \perp \left| D_{j} u_{a} \right|$, again for $i \neq j$. 

 Since expressions of the form $\left| D_{i} u_{a} \right|$ are mutually singular, the total variation of $ Du_a$ $ =$ $ \big( D_1u_a$, $D_2 u_a$, $ \ldots$, $ D_m u_a \big) $ 
 is such that 
\begin{equation}\label{evaluatesummand} 
 \left| Du_a \right|\big( (0, 1)^{m} \big) = \sum_{i=1}^{m} \left| D_i u_a \right|\big( (0, 1)^{m} \big). 
\end{equation}
 So, from \eqref{latersummand} and \eqref{evaluatesummand} together, we find that 
\begin{equation}\label{Duadomain} 
 \left| D u_a \right|\big( (0, 1)^{m} \big) = m. 
\end{equation}

 Since $\left| D_{i} u_{a} \right|$ is concentrated on $S_{i}$, we have that $\left| Du_a \right|$ is concentrated on $S = \bigcup_{i=1}^{m} S_{i}$. From \eqref{LmSi0}, we can 
 conclude that $\mathcal{L}^{m}(S) = 0$, i.e., so that $Du_a \perp \mathcal{L}^{m}$. This together with \eqref{defineDs} gives us that 
\begin{equation}\label{DsuaDua}
 D^{\operatorname{s}}u_a = Du_a. 
\end{equation}
 In turn, the equality in \eqref{DsuaDua} implies that 
\begin{equation}\label{DDm}
 \left| D^{\operatorname{s}} u_a \right|\big( (0, 1)^{m} \big) = \left| Du_a \right|\big( (0, 1)^{m} \big) = m. 
\end{equation}
 So, the equalities in \eqref{Duadomain} and \eqref{DDm} together give us the desired result. 
\end{proof}

 The integral inequality in \eqref{subgraph-lower-bound-graph} provides a key to the below proof. 

\begin{lemma}\label{beforetheorem}
 For each $a > 0$, the evaluation $ \mathcal H^{n-1}(\Gamma_a)=n $ holds for all integers \(n\geq 3\).
\end{lemma}

\begin{proof}
 Recall that $m=n-1$, and set $ \Omega=(0,1)^m$. By Lemma~\ref{afterintegral}, we have that $u_a$ is in $ BV(\Omega)$ and that $ |D^{\operatorname{s}}u_a|(\Omega) = m$. 
 Moreover, Lemma~\ref{afterintegral} gives us that $Du_a \perp \mathcal{L}^m$. Hence the absolutely continuous part of \(Du_a\) in the decomposition $$ Du_a = \nabla u_a \, 
 \mathcal{L}^m+D^{\operatorname{s}}u_a $$ vanishes, so that $ \nabla u_a=0$ holds almost everywhere on $\Omega$, with respect to $\mathcal{L}^{m}$. 
 Since $\mathcal{L}^m(\Omega)=1$, it follows that $$ \int_{\Omega}\sqrt{1+|\nabla u_a|^2}\,d\mathcal L^m + |D^{\operatorname{s}}u_a|(\Omega) = 1+m = n. $$ 
 Now $u_a$ is continuous on $\Omega$, and $\Gamma_a$ is precisely the graph of \(u_a\) over \(\Omega\). Therefore the lower bound \eqref{subgraph-lower-bound-graph}, applied 
 to $u = u_a$, gives 
\begin{equation}\label{nleqH} 
 n = \int_{\Omega}\sqrt{1+|\nabla u_a|^2}\,d\mathcal L^m + |D^{\operatorname{s}}u_a|(\Omega) \leq \mathcal H^m(\Gamma_a).
\end{equation}
 On the other hand, Lemma~\ref{oneantichain} shows that $\Gamma_a$ is an antichain in $[0,1]^n$. Hence the antichain upper bound \eqref{upperboundn} gives 
 $ \mathcal H^{n-1}(\Gamma_a)\leq n$. Since \(m=n-1\), this is the same as 
\begin{equation}\label{Hleqn} 
 \mathcal H^m(\Gamma_a)\leq n.
\end{equation}
 From \eqref{nleqH} and \eqref{Hleqn}, we obtain that $ \mathcal{H}^m(\Gamma_a)=n$, recalling that $m=n-1$. 
\end{proof}

\begin{theorem}
 For all positive integers $n$ and $k$, there exists a $k$-antichain \(A\subseteq[0,1]^n\) such that $ \mathcal H^{n-1}(A)=kn$.
\end{theorem}

\begin{proof}
 If $n = 1$, then the order relation imposed on $[0, 1]^{n}$ is linear. So, by choosing a set $A \subseteq [0, 1]$ consisting of any $k$ points in $[0, 1]$, every chain in $[0, 1]$ meets $A$ 
 in at most $k$ points, so that $\mathcal{H}^{0}(A) = k = k n$.

 For $n = 2$, this two-dimensional case is proved by Pelekis--Vlas\'ak \cite{PelekisVlasak2020}.

 Now, let $n \geq 3$, again with $m = n - 1$. We then choose distinct positive numbers $a_1$, $a_2$, $\ldots$, $a_k$, and we proceed to define 
\begin{equation}\label{defineAunion}
 A = \bigcup_{j=1}^{k} \Gamma_{a_{j}}. 
\end{equation}
 From the definition in \eqref{Phiat}, we find that: If $\text{{\bf x}} \in (0, 1)^{m}$ is fixed, then $u_{a}(\text{{\bf x}})$ is strictly increasing as a function of $a > 0$. This gives us, 
 recalling \eqref{Gammaa}, that the graphs $\Gamma_{a_{1}}$, $\Gamma_{a_{2}}$, $\ldots$, $\Gamma_{a_{k}}$ are pairwise disjoint. Also observe that: Since $\Gamma_a$ is the graph 
 of a continuous function, we have that $\Gamma_a$ is a Borel set. So, since each of $\Gamma_{a_{1}}$, $\Gamma_{a_{2}}$, $\ldots$, $\Gamma_{a_{k}}$ is a Borel set, the Hausdorff 
 measure $\mathcal{H}^{n-1}$ is additive on the finite disjoint union of the given Borel sets. So, the definition in \eqref{defineAunion} together with Lemma \ref{beforetheorem} give 
 us that $$\mathcal{H}^{n-1}(A) = \sum_{j=1}^{k} \mathcal{H}^{n-1}(\Gamma_{a_{j}}) = k n.$$ From Lemma \ref{oneantichain}, since each expression of the form $\Gamma_{a_{j}}$ 
 is an antichain, we find that: If $C \subseteq [0, 1]^{n}$ is a chain, then $C \cap \Gamma_{a_{j}}$
 contains at most one point, for an arbitrary index $j$. Consequently, the relations
 $$ \operatorname{card}(A \cap C) \leq \sum_{j=1}^{k} \operatorname{card}(C \cap \Gamma_{a_{j}}) \leq k $$
 hold. So, since $A \subseteq [0, 1]^{n}$ is such that 
 $\operatorname{card}(A \cap C) \leq k$ for every chain $C$
 and is such that $\mathcal{H}^{n-1}(A) = k n$, we obtain the desired result. 
 \end{proof}

\subsection*{Acknowledgements}
 The author acknowledges extensive interactions with GPT-5.5 Pro during the exploratory and proof-development stages of this work. All AI-generated suggestions were 
 substantially revised, corrected, and independently verified by the author, who assumes full responsibility for the mathematical content.

\end{document}